\documentclass[12pt]{amsart}
\usepackage{amssymb, latexsym}

\begin{document}
 \bibliographystyle{plain}

 \newtheorem{theorem}{Theorem}
 \newtheorem{lemma}{Lemma}
 \newtheorem{corollary}{Corollary}
 \newtheorem{problem}{Problem}
 \newcommand{\mc}{\mathcal}
 \newcommand{\rar}{\rightarrow}
 \newcommand{\Rar}{\Rightarrow}
 \newcommand{\lar}{\leftarrow}
 \newcommand{\mbb}{\mathbb}
 \newcommand{\A}{\mc{A}}
 \newcommand{\B}{\mc{B}}
 \newcommand{\cc}{\mc{C}}
 \newcommand{\D}{\mc{D}}
 \newcommand{\E}{\mc{E}}
 \newcommand{\F}{\mc{F}}
 \newcommand{\G}{\mc{G}}
 \newcommand{\HH}{\mc{H}}
 \newcommand{\I}{\mc{I}}
 \newcommand{\J}{\mc{J}}
 \newcommand{\M}{\mc{M}}
 \newcommand{\nn}{\mc{N}}
 \newcommand{\qq}{\mc{Q}}
 \newcommand{\U}{\mc{U}}
 \newcommand{\X}{\mc{X}}
 \newcommand{\Y}{\mc{Y}}
 \newcommand{\C}{\mathbb{C}}
 \newcommand{\R}{\mathbb{R}}
 \newcommand{\N}{\mathbb{N}}
 \newcommand{\Q}{\mathbb{Q}}
 \newcommand{\Z}{\mathbb{Z}}
 \newcommand{\lf}{\left\lfloor}
 \newcommand{\rf}{\right\rfloor}
 \newcommand{\dx}{\text{\rm d}x}
 \newcommand{\dy}{\text{\rm d}y}
 \newcommand{\SL}{\mathrm{SL}}
 \newcommand{\PSL}{\mathrm{PSL}}
 \newcommand{\GL}{\mathrm{GL}}
 \newcommand{\PGL}{\mathrm{PGL}}
 \newcommand{\proj}{\mathbb{P}}

\parskip=0.5ex

\title{Diophantine approximation and coloring}
\author{Alan~Haynes and Sara~Munday}
\thanks{Research supported by EPSRC grants EP/J00149X/1 and EP/L001462/1.}
\address{Department of Mathematics, University of York, UK}
\email{alan.haynes@york.ac.uk, sara.munday@york.ac.uk}

\allowdisplaybreaks


\begin{abstract}
We demonstrate how connections between graph theory and Diophantine approximation can be used in conjunction to give simple and accessible proofs of seemingly difficult results in both subjects.
\end{abstract}

\maketitle

\section{Introduction}
The goal of this paper is to bring to light some recently discovered connections between problems about graph colorings and problems about the approximation of real numbers by rationals. The connections arise as a result of the fact that many statements about the quality of approximation of real numbers can be phrased as problems about the orbits of points in certain spaces (e.g. compact metric spaces) under the action of particular groups. Once these group actions are identified, there is a correspondence between questions about the approximations and questions about the Cayley graph (to be defined below) of the given group. Information on either side of this correspondence gives information about the other. Several authors (cf. \cite{Kat2001}, \cite{RuzTuzVoi2002}, and \cite{PerSch2010}) have used this machinery to transfer information from Diophantine approximation to give upper bounds for the chromatic number of Cayley graphs. Our main result shows that interesting information about Diophantine approximation can also be obtained by going in the other direction. We show how lower bounds for the chromatic number of certain Cayley graphs can be used to give a new proof of the $p$-adic Littlewood Conjecture for quadratic irrationals, Theorem \ref{thm.mixedlit} below, which was previously proved by other authors \cite{MatTeu2004} via completely different means.

The paper is organized as follows. In Section \ref{sec.prelim} we outline the basic definitions and tools we are going to be using, and present that main problem to be solved. In Section \ref{sec.chrnum1} we present a historical overview, with proofs, of how results from Diophantine approximation can be used to give upper bounds for the chromatic number of Cayley graphs. In Section \ref{sec.chrnum2} we present the proof of our main theorem. Finally, in Section \ref{sec.furtherdirecs}, we suggest some open problems and directions for further research.

\section{Preliminaries}\label{sec.prelim}

\subsection{Cayley graphs and chromatic number}

Let $G$ be a group, possibly non-Abelian. For any subset $\A\subseteq G$, we define the {\em Cayley graph of $G$ with respect to $\A$} to be the graph having vertices indexed by the elements of $G$ and with two vertices $g,h\in G$ adjacent if and only if $h^{-1}g\in\A$ or $g^{-1}h\in\A$. For example, if we let our group be the additive group of integers and first choose $\mathcal{A}$ to be the subgroup $2\Z$, the Cayley graph obtained in that case is such that the vertex indexed by $k$ is adjacent to that indexed by $\ell$ if and only if $k\equiv \ell$ mod $2$. On the other hand, if instead we let $\mathcal{A}$ be the set $\{1, 2\}$, the Cayley graph of $G$ with respect to $\mathcal{A}$ is the graph with the vertex indexed by the integer $n$ adjacent to each of those indexed by $n-1, n+1, n-2$ and $n+2$.

A graph is said to be {\em properly colored} if each vertex is assigned a color such that no two adjacent vertices are the same color. Then, the {\em chromatic number} of a graph is the minimum number of colors (possibly infinite) needed to properly color the graph. We will denote the chromatic number of the Cayley graph of $G$ with respect to $\A$ by $\chi (G,\A).$ Of course, any finite graph has finite chromatic number. For the two examples of Cayley graphs given above, one immediately verifies that the chromatic number of the first is infinite, whereas the second has chromatic number 3. It is straightforward to see that the chromatic number of a complete graph with $n$ vertices is equal to $n$. From this it follows that if a graph $G$ contains a complete subgraph with $n$ vertices then $\chi (G)\ge n$.

\subsection{Diophantine approximation}
Suppose that $\alpha\in [0,1)$ is an irrational number and let $\alpha=[a_1,a_2,\ldots]$, where each $a_i$ is a positive integer, be its simple continued fraction expansion, with principal convergents denoted by $p_k/q_k,~k\ge 1$. In other words,
\[\alpha=[a_1, a_2, a_3, \ldots]:= \frac{1}{a_1 + \frac{1}{a_2+\frac{1}{a_3 + \cdots}}},\]
and $p_k/q_k=[a_1, \ldots, a_k]$. It is useful to keep in mind the recurrence relations that define the convergents: we have, for each $n\in\N$, that \[p_{n+1}= a_{n+1} p_{n} + p_{n-1}\  \text{ and }\  q_{n+1} =a_{n+1} q_{n} + q_{n-1},\]
where we define $p_0:=0$ and $q_0:=1$.
These relations can be proved by induction (see, for instance \cite{cf}).  Now, for each $k\in\N$ let $A_k\in \GL_2(\Z)$ be defined by
\begin{equation}\label{eqn.cf1}
A_k:=\left(\begin{array}{cc}0 & 1 \\ 1 & a_{k} \\ \end{array}\right),
\end{equation}
and note that, by the recurrence relation defining the convergents,
\begin{equation}\label{eqn.cf2}
A_k\cdots A_1\genfrac{(}{)}{0pt}{}{0}{1}=\genfrac{(}{)}{0pt}{}{q_{k-1}}{q_{k}}.
\end{equation}
One of the most important facts about continued fractions is that the principal convergents give excellent approximations to $\alpha$. It is not difficult to prove that for any $k\in\N$,
\begin{equation}\label{eqn.cf0}
\left|\alpha-\frac{p_k}{q_k}\right|\le\frac{1}{q_kq_{k+1}}.
\end{equation}
Multiplying through by $q_k^2,$ this implies that
\[q_k|q_k\alpha-p_k|\le 1,\]
and, with the notation $\|\cdot\|$ denoting the distance to the nearest integer,  it follows from the preceding inequality that for any irrational $\alpha$ there are infinitely many integers $n$ satisfying
\begin{equation}\label{eqn.weakDirichlet}
n\|n\alpha\|\le 1.
\end{equation}
This is a slightly weakened form of Dirichlet's Theorem in Diophantine approximation (see \cite{cf}). It is natural to inquire as to whether or not this conclusion can be improved. It turns out that the answer is yes, but not by much, if we still want a result which holds for all irrational numbers. It is a theorem of Hurwitz that for any irrational $\alpha$ there are infinitely many integers $n$ for which
\begin{equation*}
n\|n\alpha\|\le 1/\sqrt{5},
\end{equation*}
and the constant $1/\sqrt{5}$ cannot in general be replaced by anything smaller. If we were willing to throw out a set of real numbers of Lebesgue measure zero, further improvements to Hurwitz's Theorem would lead us into the maze of the metric theory of Diophantine approximation, where we would encounter the Borel-Bernstein Theorem, Khintchine's Theorem, and the Duffin-Schaeffer Conjecture. However this is not our present goal, so we refer the reader interested in the metric theory to \cite{Har1998}, and we choose instead to steer toward the topic of multiplicative approximation.

There is a famous conjecture in Diophantine approximation called the Littlewood Conjecture, which states that for all pairs of real numbers $\alpha$ and $\beta$, we have that
\begin{equation}\label{eqn.littconj}
\inf_{n\in\N}n\|n\alpha\|\|n\beta\|=0.
\end{equation}
If this conjecture were true, it would assert that with the help of an extra factor $\|n\beta\|$, the quantity on the left hand side of (\ref{eqn.weakDirichlet}) would become arbitrarily small infinitely often. This is a statement about how well pairs of real numbers can be simultaneously approximated (in a multiplicative sense) by rational numbers with the same denominator. Although there are many partial results, the Littlewood Conjecture remains an open problem. We point out for later reference that (\ref{eqn.littconj}) is not even known to hold in some seemingly simple cases, for example when $\alpha$ and $\beta$ are quadratic irrationals (real roots of irreducible quadratic polynomials over $\Q$) which are linearly independent over $\Q$ (for instance, if $\alpha=\sqrt{2}$ and $\beta=\sqrt{3}$).

More recently, a variant of the Littlewood Conjecture has been proposed by de Mathan and Teuli\'{e} \cite{MatTeu2004}, in which the factor $\|n\beta\|$ is replaced by the {\em $p$-adic absolute value} of $n$ for some prime $p$. The $p$-adic absolute value of $n\in\Z$, which is denoted by $|n|_p$, is equal to $0$ if $n=0$ and is equal to $p^{-\ell}$ if $n=p^\ell m$ for some non-zero integer $m$ coprime to $p$. For the interested reader, more can be found about the $p$-adic absolute values in Chapter 1 of \cite{Koblitz}, but for our purposes, since we will only be working with integers, this is all we are going to need to know. As a further guide to intuition notice that the statement
\[|n|_p\le p^{-k}\]
is equivalent to the statement
\[n=0~\mathrm{mod}~p^k.\]
We leave it to the reader to verify that if $m, n\in\Z$, then
\begin{equation}\label{eqn.p-adictriineq}
|n+m|_p\leq \max\{|n|_p, |m|_p\}\leq |n|_p+|m|_p.
\end{equation}
The $p$-adic Littlewood Conjecture is the assertion that, for any real number $\alpha$ and for any prime $p$, we have
\begin{equation}\label{eqn.plittconj}
\inf_{n\in\N}n|n|_p\|n\alpha\|=0.
\end{equation}
There are many similarities between the theory surrounding the Littlewood Conjecture and the theory surrounding the $p$-adic Littlewood Conjecture. However, although both are still open problems, there are some hints that the $p$-adic conjecture may be in some ways easier. One indication of this is the following theorem.
\begin{theorem}\label{thm.mixedlit}
For any prime $p$ and for any quadratic irrational $\alpha\in\R$,
\[\liminf_{n\in\N}n|n|_p\|n\alpha\|\log n<\infty.\]
\end{theorem}
This theorem was first proved by de Mathan and Teuli\'{e} in \cite{MatTeu2004}, and in Section \ref{sec.chrnum2} we will give our own proof which illustrates the connection between chromatic number of Cayley graphs of certain groups and dynamics or orbits of points in spaces on which the groups act.

At this point, to  conclude this section, let us give an indication of the reason that quadratic irrational numbers have been making an appearance. One thing that makes these numbers special is a theorem due to Lagrange (Theorem 28 in \cite{cf}), which states that the simple continued fraction expansion of an irrational number $\alpha$ is {\em eventually periodic} if and only if $\alpha$ is a quadratic irrational. By this, we mean that a number is a quadratic irrational if and only if its continued fraction expansion has the form $[b_1, \ldots, b_r, \overline{a_{1}, \ldots, a_{s}}]$, where the line over the block $a_{1}, \ldots, a_{s}$ means that this block is repeated infinitely many times.

\section{Bounding chromatic number of Cayley graphs}\label{sec.chrnum1}
In this section we will begin to understand the connections between chromatic number, dynamical systems, and Diophantine approximation. We would like to emphasize that the ideas in this section were originally discovered by other authors. We will give references to their work, but for completeness we also provide proofs of all relevant results. To start, consider the following problem, posed by Erd\H{o}s in 1975.
\begin{problem}\label{prob.erd1}
Suppose that $\{n_k\}_{k\in\N}$ is a sequence of positive integers satisfying the condition
\begin{equation}\label{eqn.lacdef}
\frac{n_{k+1}}{n_k}\ge\lambda>1 \text{ for some }\lambda\text{ and all }k\in\N.
\end{equation}
Does there exist a number $\alpha\in\R/\Z$ such that the collection of points $\{n_k\alpha\}_{k\in\N}$ is not dense in $\R/\Z$?
\end{problem}
For future reference, a sequence of integers satisfying the growth condition (\ref{eqn.lacdef}) is called {\em lacunary}.

It turns out that this problem was solved almost 50 years before it was stated, by Khintchine (see Hilfssatz III in \cite{khin}). The statement proved by Khintchine is actually a little more general, as instead of the lacunarity condition, he proves that if there exists a $\lambda>1$ and $t\in\N$ such that ${n_{k+t}}/{n_k}\ge\lambda>1$ for all $k\in \N$, then there exists a point $\alpha$ for which the collection of points $\{n_k\alpha\}_{k\in\N}$ is not dense in $\R/\Z$.

Problem \ref{prob.erd1} was also settled independently by de Mathan and Pollington \cite{Mat1980,Pol1979}. They proved that if $\{n_k\}\subseteq\N$ is lacunary then the collection of $\alpha$ for which $\{n_k\alpha\}$ is not dense in $\R/\Z$ is uncountable (and in fact that it has Hausdorff dimension $1$). To give some feeling for how their proof goes we prove the following weaker result, which is sufficient both to answer Problem \ref{prob.erd1} and for our application. The proof below can be reasonably easily adapted (by keeping track of the possible number of intervals constructed at each stage) to give the proof of the stronger Hausdorff dimension result.
\begin{lemma}\label{lem.nondenseorbit}
Suppose that $\{n_k\}_{k\in\N}\subseteq\N$ is lacunary. Then there exists an integer $N\in\N$ and a real number $\alpha$ such that
\[\|n_k\alpha\|>N^{-1}~\text{for all}~k\in\N.\]
\end{lemma}
\begin{proof}
Our proof is a simplified and expanded version of Pollington's \cite{Pol1979}, although keep in mind that his conclusion is much stronger than ours. Before we jump into the full details, first consider how we could argue if $\{n_k\}$ satisfied (\ref{eqn.lacdef}) with $\lambda>2$. In that case, select $N\in\N$ large enough that
\begin{equation}\label{eqn.lemprf1}
\lambda(1-2N^{-1})>2,
\end{equation}
and choose any integer $0\le b_1< n_1$. Then any real number $x$ in the interval
\[I_1:=\left(\frac{b_1+N^{-1}}{n_1},\frac{b_1+1-N^{-1}}{n_1}\right)\]
satisfies the inequality
\[\|n_1x\|>N^{-1}.\]
We are almost finished, because now we can continue inductively, choosing for each $k\ge 1$ an integer $b_k$ and an interval
\[I_k:=\left(\frac{b_k+N^{-1}}{n_k},\frac{b_k+1-N^{-1}}{n_k}\right),\]
with $I_k\subseteq I_{k-1}$. The reason we can guarantee this last condition is that, by the lacunarity of our sequence and by (\ref{eqn.lemprf1}),
\[\left\lfloor n_k\cdot|I_{k-1}|\right\rfloor=\left\lfloor\frac{n_k(1-2N^{-1})}{n_{k-1}}\right\rfloor>2,\]
so each interval $I_{k-1}$ contains at least two fractions with denominator $n_k$. Now since the $I_k$ are nested intervals, their intersection contains a point, which we call $\alpha$. By construction $\alpha\in I_k$ for all $k$, which allows us to conclude that
\[\|n_k\alpha\|>N^{-1}~\text{for all}~k.\]

It is instructive to examine why this proof breaks down in the general case of $\lambda>1$. If $\lambda$ is close to $1$ then, by choosing $N$ large enough and taking $I_1$ as above, we can guarantee that there is a fraction with denominator $n_2$ in $I_1$. However, we can only guarantee the existence of one such fraction, not two as before, so when we come to choose the interval $I_2$ nested inside $I_1$, we are no longer able to ensure that the length will be $(1-2N^{-1})/n_2$, as it was previously, but it may have to be significantly shorter.
If this worst-case scenario continues then eventually we could be forced into a position where our interval $I_k$ is smaller than $2N^{-1}/n_{k+1}$. Then it is game over, because if there is a fraction with denominator $n_{k+1}$ close to the center of $I_k$, we will be forced to throw away all of $I_k$ at the next step.

To get around this problem we will inductively construct our nested intervals by looking ahead at how the fractions corresponding to the next several denominators chop up the interval we are in, and choosing the largest resulting subinterval. There is one technical assumption that we need to make, which is that the integers $n_k$ do not grow at a rate that is faster than lacunary. This does not involve any loss of generality in our proof because if
\[n_{k+1}>\lambda(\lambda+1)n_k,\]
then we can insert the integer $n_k':=\lceil\lambda n_k\rceil$ into our sequence. Since
\[\frac{n_k'}{n_k}\ge \lambda\quad\text{ and }\quad\frac{n_{k+1}}{n_k'}\ge\frac{n_{k+1}}{\lambda n_k+1}\ge \frac{n_{k+1}}{(\lambda+1)n_k}>\lambda,\]
our new sequence still satisfies the lacunarity condition (\ref{eqn.lacdef}). Furthermore, if a real number $\alpha$ satisfies the conclusion of our lemma for the new sequence, then it will also satisfy the conclusion for our original sequence (since it is a subsequence of the new sequence). Therefore, by continuing to insert integers into our sequence where needed we may assume, after relabeling and without loss of generality, that
\begin{equation}\label{eqn.uppandlowbnd}
\lambda\le \frac{n_{k+1}}{n_k}<\lambda(\lambda+1)~ \text{ for all }~k\in\N.
\end{equation}
As we said, this is a harmless technical assumption, which will be used in only one place below.

In this paragraph we focus on the main line of the argument and in the next paragraph we fill in the choices of the parameters involved. Start with an interval $I_1$, chosen to avoid points too close to fractions with denominator $n_1$, and having length $\Delta/n_1$ for some small $\Delta>0$. Then look at fractions with denominators $n_2,n_3,\ldots ,n_{1+K}$, for some integer $K$, to be determined, and around each of these which falls in $I_1$ remove a small interval of radius $N^{-1}/n_1$, where $N$ is also to be determined. If we choose $\Delta$ small enough then we can guarantee that at most one fraction with each denominator $n_2,\ldots , n_{1+K}$ is contained in $I_1$, therefore after removing all of the small intervals centered at these fractions we will be left with a subset $J_1\subseteq I_1$ which is a union of at most $K+1$ intervals and which has measure greater than or equal to
\[|I_1|-\frac{2KN^{-1}}{n_1}=\frac{\Delta-2KN^{-1}}{n_1}.\]
The largest component interval of $J_1$ will have length at least
\begin{eqnarray}\label{*1}\frac{\Delta-2KN^{-1}}{(K+1)n_1},\end{eqnarray}
and assuming that $K$ and $N$ have been chosen large enough (depending upon $\lambda$), we can ensure that this is larger than
\begin{eqnarray}\label{*2}\frac{\Delta}{\lambda^Kn_1}.\end{eqnarray}
Since $n_{1+K}>\lambda^Kn_1,$ this guarantees that we can choose our next subinterval $I_{1+K}$ (we skipped some indices on purpose for notational convenience) to be a subinterval of $J_1$ of length equal to
\[\frac{\Delta}{n_{1+K}}.\]
Finally we can apply the inductive hypothesis and repeat exactly the same argument, with $n_1$ replaced by $n_{1+K}$ and with the same values of $\Delta, K,$ and $N$, to construct a sequence of nested intervals $\{I_{1+jK}\}_{j\in\N}$ with every point $x$ in $I_{1+jK}$, for each $j$, satisfying
\[\|n_kx\|>N^{-1}~\text{ for all}~1\le k\le 1+jK.\]
The intersection of these intervals will be a real number $\alpha$ which satisfies the conclusion of the lemma.

This argument will become a bona fide proof if we can specify all of the parameters involved. This is the end game, but of course if the end game can't be played to our advantage then we're still in trouble. Let us recall the crucial ingredients: we need to be able to choose $\Delta$ sufficiently small that there is at most one fraction with each denominator $\{n_2, n_3, \ldots, n_{1+K}\}$ in $I_1$, and we must choose $K$ (and then $N$) sufficiently large that $(\ref{*1})>(\ref{*2})$.  So, with the second constraint in mind, we begin by choosing $K\in\N$ so that
\[K+1<\lambda^K.\]
Then, to satisfy the first constraint,  choose $\Delta>0$ so that
\begin{equation*}
\Delta<(\lambda(\lambda+1))^{-K}.
\end{equation*}
In view of our assumption (\ref{eqn.uppandlowbnd}), for each $1\le k\le K$ the number of fractions with denominator $n_{1+k}$ which lie in $I_1$ will be at most
\[\left\lfloor|I_1|n_{1+k}\right\rfloor+1=\left\lfloor\frac{\Delta n_{1+k}}{n_1}\right\rfloor+1<\left\lfloor\Delta (\lambda(\lambda+1))^k\right\rfloor+1=1,\]
so our first constraint is taken care of. Now the second constraint will be satisfied provided we choose $N\in\N$ large enough so that
\[\Delta\left(1-\frac{(K+1)}{\lambda^K}\right)>2KN^{-1},\]
which is clearly possible. Finally, one more important point, we want to ensure that we can actually choose
\[I_1\subseteq \left(\frac{b_1+N^{-1}}{n_1},\frac{b_1+1-N^{-1}}{n_1}\right)\]
to be an open interval of length
\[\frac{\Delta}{n_1},\]
so we should also make sure that $N$ has been chosen large enough so that
\[\Delta<1-2N^{-1}.\]
With this choice of parameters, returning to the argument in the previous paragraph completes the proof.
\end{proof}

There is another problem of Erd\H{o}s from 1987 which, at first glance, seems to be quite different.
\begin{problem}\label{prob.erd2}
Suppose $\{n_k\}_{k\in\N}\subseteq\N$ is a lacunary sequence. Does the Cayley graph of $(\Z,+)$ with respect to $\{n_k\}$ have finite chromatic number?
\end{problem}
This problem was solved by Katznelson \cite{Kat2001}, who also pointed out that it is closely related to Problem \ref{prob.erd1}. To see the connection, consider the following lemma.
\begin{lemma}\label{lem.chrnum<->recurrence1}
For any sequence of integers $\A=\{n_k\}$, finite or infinite, if there exists an $\alpha\in\R/\Z$ and an $N\in\N$ such that
\begin{equation}\label{eqn.misszero}
\|n_k\alpha\|>N^{-1}~\text{for all}~k\in\N,
\end{equation}
then
\[\chi(\Z,\A)\le N.\]
\end{lemma}
\begin{proof}
Suppose that $\alpha$ and $N$ satisfy the hypotheses of the lemma, and for each $1\le n\le N$ define $B_n\subseteq\R/\Z$ by
\[B_n:=\left[\frac{n-1}{N},\frac{n}{N}\right).\]
The sets $\{B_n\}_{1\le n\le N}$ form a partition of $\R/\Z$.

We define a coloring of the Cayley graph of $\Z$ with respect to $\A$ by the rule that an integer $m$ is assigned color $n$ if and only if $m\alpha~\mathrm{mod}~1\in B_n$. We claim that this is a proper coloring. To see this suppose that $m_1$ and $m_2$ are two integers which are both assigned the same color. Then we have that
\[\|(m_1-m_2)\alpha\|=\|m_1\alpha-m_2\alpha\|\le 1/N,\]
which in light of (\ref{eqn.misszero}) implies that $|m_1-m_2|\not\in\A$. In other words $m_1$ and $m_2$ are not adjacent in the Cayley graph of $\Z$ with respect to $\A$ and our claim, as well as the conclusion of the lemma, are verified.
\end{proof}
Combining Lemma \ref{lem.chrnum<->recurrence1} and Lemma \ref{lem.nondenseorbit} shows that if $\A=\{n_k\}_{k\in\N}$ is lacunary then $\chi(\Z,\A)<\infty$. This yields an affirmative answer to Problem \ref{prob.erd2}.

\section{Going the other direction}\label{sec.chrnum2}
In the previous section we saw an example of how results from Diophantine approximation can be used to produce upper bounds for the chromatic number of Cayley graphs. Now we want to turn this around and see if knowledge of the chromatic number can tell us anything interesting about Diophantine approximation. To begin with, we reformulate Lemma \ref{lem.chrnum<->recurrence1} in the following equivalent form.
\begin{lemma}\label{lem.chrnum<->recurrence2}
For any sequence of integers $\A=\{n_k\}$, finite or infinite, if for some $N\in\N$
\[\chi(\Z,\A)>N,\]
then for any $\alpha\in\R/\Z$, there exists a $k$ such that
\[\|n_k\alpha\|\le N^{-1}.\]
\end{lemma}
As an application, let $M$ be a positive integer and $\A=\{1,2,\cdots ,M\}.$ In the Cayley graph of $\Z$ with respect to $\A$, if we look at any collection of $M+1$ vertices represented by consecutive integers, they will all be adjacent. In other words the Cayley graph contains complete subgraphs with $M+1$ vertices, which implies that
\[\chi(\Z,\A)\ge M+1.\]
By Lemma \ref{lem.chrnum<->recurrence2} we conclude that for any $\alpha\in\R/\Z$, there is an integer $1\le m\le M$ such that
\[\|m\alpha\|\le M^{-1}.\]
Since $M$ is arbitrary, this is essentially the weaker form of Dirichlet's Theorem which we stated above.

To obtain a more interesting application, we will extend the idea behind Lemma \ref{lem.chrnum<->recurrence2} to a slightly more abstract setting, namely, that of groups acting {\em isometrically} on metric spaces. For a left action of a group $G$ on a metric space $(X,d)$ to be  isometric means that we have $d(gx, gy)=d(x,y)$ for all $x, y\in X$ and for all $g\in G$. For each $N\in\N$ we define $\B(N)$ to be the minimum number of closed balls of diameter $N^{-1}$ needed to cover $X$. Consider the following generalization of Lemma \ref{lem.chrnum<->recurrence2}.
\begin{lemma}\label{lem.chrnum<->recurrence3}
Suppose that $G$ acts isometrically on $(X, d)$ and, as before, let $\mathcal{A}$ be a subset of $G$. Also suppose that $\B (N)$ is finite for all $N$. If, for some $N\in\N$,
\[\chi (G,\A)> \B (N),\]
then for any $x\in X$, there exists an element $g\in\A$ satisfying
\[d(gx,x)\le N^{-1}.\]
\end{lemma}
\begin{proof}
Let us prove the contrapositive, which is the statement that if, for some $N\in\N$, there exists an $x\in X$ such that
\[d(gx,x)>N^{-1}~\text{for all}~g\in\A,\]
then we have that
\[\chi (G,\A)\le \B (N).\]
Suppose the hypothesis of this statement is satisfied and partition $X$ into $\B(N)$ disjoint sets $B_1,\ldots ,B_{\B(N)}$, each of diameter less than or equal to $N^{-1}$. Then color the elements of $G$ by the rule that $g\in G$ has color $1\le n\le \B(N)$ if and only if $gx\in B_n$. This is a proper coloring of the Cayley graph of $G$ with respect to $\A$, since if $g,h\in G$ both have color $n$ then $gx,hx\in B_n$, which implies that
\[d(gx,hx)=d(h^{-1}gx,x)=d(g^{-1}hx,x)\le N^{-1},\]
and therefore, by assumption, that $h^{-1}g,g^{-1}h\not\in\A$.
\end{proof}
In the following subsection, we will show how this lemma can be used to prove Theorem \ref{thm.mixedlit}.
\subsection{Our metric space and group action}
First we describe the metric space and group action which we will use, in conjunction with Lemma \ref{lem.chrnum<->recurrence3}, to prove Theorem \ref{thm.mixedlit}. Fix a prime $p$ and consider the collection $X$ of integer points visible from the origin,
\[X:=\left\{\begin{pmatrix}n_1\\n_2\end{pmatrix}\in\Z^2:\mathrm{gcd}(n_1,n_2)=1\right\},\]
with $d:X\times X\rar\R$ defined by
\[d(\mathbf{m},\mathbf{n}):=|m_1n_2-m_2n_1|_p.\]
For anyone who is already in the know and wants to bypass the rest of this paragraph, this metric space is one-dimensional projective space over $\Q$, with the $p$-adic version of the usual metric. Otherwise, notice that the function $d$ clearly satisfies $d(\mathbf{m},\mathbf{n})\ge 0$, with equality if and only $\mathbf{m}=\mathbf{n}$, and it also satisfies $d(\mathbf{m},\mathbf{n})=d(\mathbf{n},\mathbf{m})$ for all $\mathbf{m},\mathbf{n}\in X$. Therefore to show that $(X,d)$ is a metric space we just need to check that $d$ satisfies the triangle inequality. Suppose that $\mathbf{k},\mathbf{m},$ and $\mathbf{n}$ are in $X$. Since $\mathrm{gcd}(k_1,k_2)=1,$ at least one of $k_1$ and $k_2$ must be coprime to $p$, and we will assume that it is $k_1$ (the argument in the other case is the same). Then $|k_1|_p=1$ and we have that
\begin{align*}
d(\mathbf{m},\mathbf{n})&=|m_1n_2-m_2n_1|_p\\
&=|k_1|_p\cdot |m_1n_2-m_2n_1|_p\\
&=|k_1m_1n_2-k_1m_2n_1|_p.
\end{align*}
Now by adding and subtracting $k_2m_1n_1$ and rearranging, this is equal to
\[|(k_2m_1n_1-k_1m_2n_1)+(k_1m_1n_2-k_2m_1n_1)|_p,\]
which by the triangle inequality for the $p$-adic absolute value (\ref{eqn.p-adictriineq}) is
\begin{align*}
&\le|k_2m_1n_1-k_1m_2n_1|_p+|k_1m_1n_2-k_2m_1n_1|_p\\
&=|n_1|_p\cdot d(\mathbf{m},\mathbf{k})+|m_1|_p\cdot d(\mathbf{k},\mathbf{n}).
\end{align*}
Since each of $|n_1|_p$ and $|m_1|_p$ are at most $1$, the right hand side is bounded above by
\[d(\mathbf{m},\mathbf{k})+d(\mathbf{k},\mathbf{n}),\]
and this proves the triangle inequality for $d$.

Now that we know that $(X,d)$ is a metric space, we would like to find good upper bounds for the quantities $\B (N)$. Choose $N\in\N$ and let $k\ge 0$ be the unique integer satisfying
\begin{equation}\label{eqn.Nbnd}
p^{-k}\le N^{-1}< p^{-k+1}.
\end{equation}
Next, for each integer $1\le \ell\le p^k$ define subsets $A_{1,\ell},A_{2,\ell}\subseteq X$ by
\begin{align*}
A_{1,\ell}&:=\{\mathbf{n}\in X:p\nmid n_1~\text{and}~n_1^{-1}n_2=\ell ~\mathrm{mod}~ p^k\},\quad\text{and}\\
A_{2,\ell}&:=\{\mathbf{n}\in X:p| n_1~\text{and}~n_1n_2^{-1}=\ell ~\mathrm{mod}~ p^k\}.
\end{align*}
Every point in $X$ belongs to at least one of these sets, for some $\ell$. Furthermore if for some $\ell$ we have that $\mathbf{m}$ and $\mathbf{n}$ are both in $A_{1,\ell}$, then
\[m_1^{-1}m_2=n_1^{-1}n_2~\mathrm{mod}~p^k\]
which implies that
\[m_1n_2-m_2n_1=0~\mathrm{mod}~p^k,\]
or in other words that
\[d(\mathbf{m},\mathbf{n})\le p^{-k}.\]
The same conclusion holds if $\mathbf{m}$ and $\mathbf{n}$ are both in $A_{2,\ell}$ for some $\ell$. From each set $A_{1,\ell}$ and from each non-empty set $A_{2,\ell}$ (some of these are actually empty) we choose an element of $X$, and we label the resulting collection of elements $\{\mathbf{m}_i\}_{i=1}^M$. Then it follows that $M\le 2p^k$, and by construction the collection of closed balls
\[\{B(\mathbf{m}_i,p^{-k})\}_{i=1}^M\]
covers all of $X$. Therefore, by virtue of (\ref{eqn.Nbnd}), the collection of balls
\[\{B(\mathbf{m}_i,N^{-1})\}_{i=1}^M\]
also covers $X$, and we conclude from (\ref{eqn.Nbnd}) again that
\begin{equation}\label{eqn.B(N)bnd}
\B (N)\le 2p^k< 2pN.
\end{equation}
It turns out that this bound is best possible, up to determination of the best constant on the right hand side of the inequality, but even without the best constant it is sufficient for our application.

For our group action, let $G=\mathrm{GL}(2,\Z)$ be the multiplicative group of $2\times 2$ matrices with integer coefficients and determinant $\pm 1$. For any $A\in G$ and $\mathbf{m}\in X$, we define $A\mathbf{m}$ to be the vector obtained by multiplying $\mathbf{m}$ on the left by the matrix $A$. This is a natural thing to do, but we must verify that it does indeed define a group action. The vector $A\mathbf{m}$ is clearly another element of $\Z^2$, and to see that is an element of $X$ write $A=(a_{ij})$, so that
\[A\mathbf{m}=\begin{pmatrix}n_1 \\ n_2 \end{pmatrix},\]
with
\[n_1=a_{11}m_1+a_{12}m_2~\text{and}~n_2=a_{21}m_1+a_{22}m_2.\]
Let $e=\mathrm{gcd}(n_1,n_2)$. In order to show that $A\mathbf{m}\in X$, we must show that $e=1$.  Since $e$ divides $n_1$ and $n_2$ it also divides
\[a_{22}n_1-a_{12}n_2=(a_{11}a_{22}-a_{12}a_{21})m_1=\pm m_1,\]
as well as
\[a_{11}n_2-a_{21}n_1=(a_{11}a_{22}-a_{12}a_{21})m_2=\pm m_2.\]
Therefore $e=1$ and $A\mathbf{m}\in X$. The rest of the hypotheses defining a group action (i.e., associativity and identity) follow from basic properties of matrix multiplication, and this verifies that the natural action of $G$ on $X$ does genuinely produce a group action.

Finally, to see that this action is isometric we must show that for any $A\in G$ and for any $\mathbf{m}$ and $\mathbf{n}$ in $X$,
\begin{equation}\label{eqn.isomdef2}
d(\mathbf{m},\mathbf{n})=d(A\mathbf{m},A\mathbf{n}).
\end{equation}
Let us use the notation $\begin{pmatrix}\mathbf{m}&\mathbf{n}\end{pmatrix}$ to denote the $2\times 2$ matrix whose columns are the vectors $\mathbf{m}$ and $\mathbf{n}$. Now observe that, from the definition of $d$,
\[d(\mathbf{m},\mathbf{n})=\left|\mathrm{det}\begin{pmatrix}\mathbf{m}&\mathbf{n}\end{pmatrix}\right|_p.\]
On the other hand,
\[d(A\mathbf{m},A\mathbf{n})=\left|\mathrm{det}\begin{pmatrix}A\mathbf{m}&A\mathbf{n}\end{pmatrix}\right|_p=\left|\mathrm{det}\left(A\cdot\begin{pmatrix}\mathbf{m}&\mathbf{n}\end{pmatrix}\right)\right|_p.\]
Since the determinant is multiplicative and since $A\in G$ (and thus has determinant $\pm 1$), these two equations combine to prove (\ref{eqn.isomdef2}) and therefore that $G$ acts isometrically on $X$.

\subsection{Proof of Theorem \ref{thm.mixedlit}}
Now we are ready for the proof of the main theorem. We may assume without loss of generality that $\alpha\in[0,1)$, and we write $\alpha=[b_1,\ldots ,b_r,\overline{a_1,\ldots ,a_s}]$ for its continued fraction expansion. Using the notation of (\ref{eqn.cf1}), define $A,B\in G$ by
\[A=A_{r+s}\cdots A_{r+1}~\text{ and }~B=A_r\cdots A_1.\]
For each $N\in\N$ let $\A_N\subseteq G$ be the set
\[\A_N=\{B^{-1}A^iB:1\le i\le N\},\]
and note that the restriction to $\A_N$ of the Cayley graph of $G$ with respect to $\A_N$ is a complete graph on $N$ vertices. Therefore we must have that $\chi (G,\A_N)\ge N> \B(N/2p)$ by (\ref{eqn.B(N)bnd}). It follows from Lemma \ref{lem.chrnum<->recurrence3} that, for some $1\le i\le N$,
\[d\left(B^{-1}A^{i}B\genfrac{(}{)}{0pt}{}{0}{1},\genfrac{(}{)}{0pt}{}{0}{1}\right)\le\frac{2p}{N}.\]
Restating this in light of the connection to continued fractions (equation (\ref{eqn.cf2})), for any $N\in\N$, there exists an integer $i$ with $1\le i\le N$, such that
\begin{equation}\label{eqn.n_Ninequality}
|b_{22}q_{is+r-1}-b_{12}q_{is+r}|_p\le \frac{2p}{N}.
\end{equation}
Now for each $N$ let $i$ be such an integer and set
\[n_N=|b_{22}q_{is+r-1}-b_{12}q_{is+r}|.\]
Note that $b_{22}$ and $b_{12}$ are coprime, as are $q_{is+r-1}$ and $q_{is+r}$, hence (using the easily verified fact that the difference between denominators of consecutive convergents must tend to infinity) we have that $n_N\not=0$ for $N$ large enough. Since the right hand side of (\ref{eqn.n_Ninequality}) tends to $0$ as $N\rar\infty$, we must also have that $n_N\rar\infty$.

Next let $b=\max (|b_{22}|,|b_{12}|)$. Then, using (\ref{eqn.cf0}), we have that
\begin{align*}
n_N\|n_N\alpha\|&\le 2bq_{is+r}\|n_N\alpha\|\\
&\le 2bq_{is+r}\left(|b_{22}|\cdot\|q_{is+r-1}\alpha\|+|b_{12}|\cdot\|q_{is+r}\alpha\|\right)\\
&\le 4b^2.
\end{align*}
Finally, since the partial quotients in the continued fraction expansion of $\alpha$ are bounded, there exists a constant $C>0$ such that
\[n_N\le C^N~\text{ for all }~N\in\N.\]
These observations together imply that
\[n_N|n_N|_p\|n_N\alpha\|\log n_N\le 8b^2p\log C~\text{ for all }~N\in\N,\]
which verifies the statement of the theorem.

\section{Directions for further research}\label{sec.furtherdirecs}
In this paper we have aimed to present a collection of ideas and results which highlight connections between Diophantine approximation, dynamical systems, and chromatic number. We hope that the reader has enjoyed these connections as much as we have, and we would like to conclude with several open questions for further research:
\begin{enumerate}
\item[(i)] Can the chromatic number proof of Theorem \ref{thm.mixedlit} be extended to accommodate numbers $\alpha$ which are not quadratic irrationals?
\item[(ii)] Can the idea of the proof be used to prove the Littlewood Conjecture for the case when $\alpha$ and $\beta$ are (linearly independent) quadratic irrationals?
\item[(iii)] More modestly, can the idea of the proof be used to prove that there is an $k\ge 2$ such that for any collection $\alpha_1,\ldots ,\alpha_k$ of quadratic irrationals,
    \[\inf_{n\in\N}n\|n\alpha_1\|\cdots\|n\alpha_k\|=0?\]
\end{enumerate}
Questions (ii) and (iii) are particularly attractive, and it is perhaps instructive to return to the setup and proof of Theorem \ref{thm.mixedlit} above to see where it breaks down and what one might do to fix it.


\begin{thebibliography}{99}

\bibitem{Har1998}
G.~Harman: \emph{Metric number theory}, LMS Monographs New Series, vol.~18, Clarendon Press, 1998.

\bibitem{Kat2001} Y.~Katznelson: {\em Chromatic numbers of Cayley graphs on $\Z$ and recurrence}, Combinatorica  21  (2001),  no. 2, 211-219.

\bibitem{khin} I.~Ya.~Khintchine: {\em \"{U}ber eine klasse linearer Diophantischer approximationen}, Rend. Circ. Mat. Palermo 50 (1926), 170-195.

\bibitem{cf} I.~Ya.~Khintchine: {\em Continued Fractions}, Univ. Chicago Press, 1964.


\bibitem{Koblitz} N. Koblitz: {\em $p$-adic Numbers, $p$-adic Analysis, and Zeta-Functions}, 2nd ed., Graduate Texts in Mathematics 58, Springer-Verlag, New York, 1984.

\bibitem{Mat1980} B.~de Mathan: {\em Numbers contravening a condition in density modulo 1}, Acta Math. Acad. Sci. Hungar.  36 (1980),  no. 3-4, 237-241.

\bibitem{MatTeu2004} B.~de Mathan and O.~Teuli\'{e}: \emph{Probl\`{e}mes Diophantiens simultan\'{e}s}, Monatsh. Math. 143 (2004), 229-245.

\bibitem{PerSch2010} Y.~Peres and W.~Schlag: {\em Two Erd\H{o}s problems on lacunary sequences: chromatic number and
 Diophantine approximation}, Bull. Lond. Math. Soc.  42  (2010),  no. 2, 295-300.

\bibitem{Pol1979} A.~D.~Pollington: \emph{On the density of the sequence $\{n_k\xi\}$},  Illinois J. Math. 23 (1979), no. 4, 511-515.

\bibitem{RuzTuzVoi2002} I.~Z.~Ruzsa, Zs.~Tuza, and M.~Voigt: {\em Distance graphs with finite chromatic number}, J. Combin. Theory Ser. B  85  (2002),  no. 1, 181-187.

\end{thebibliography}
\end{document}